\documentclass[12pt,reqno]{amsart}
\usepackage{amssymb}
\usepackage{amsfonts}
\usepackage{amsmath}
\usepackage{amsthm}
\usepackage{bm}
\usepackage{color}
\usepackage{enumerate}
\usepackage{epsfig}
\usepackage{hyperref}
\makeatletter
\newcommand{\Aut}{\mathrm{Aut}}
\newcommand{\Cay}{\mathrm{Cay}}

\newcommand{\Sy}{\mathrm{S}}
\newtheorem{thm}{Theorem}[section]
\newtheorem{cor}[thm]{Corollary}
\newtheorem{pro}[thm]{Proposition}
\newtheorem{lem}[thm]{Lemma}

\begin{document}

\title[Isomorphisms of Cayley digraphs]
{On Isomorphisms of tetravalent Cayley digraphs over dihedral groups}%

\author{Jin-Hua Xie}
\address{Jin-Hua Xie, Center for Combinatorics, LPMC, Nankai University, Tianjin 300071, China}
\email{jinhuaxie@nankai.edu.cn (J.-H. Xie)}

\author{Zai Ping Lu}
\address{Zai Ping Lu, Center for Combinatorics, LPMC, Nankai University, Tianjin 300071, China}
\email{lu@nankai.edu.cn (Z.P. Lu)}

\author{Yan-Quan Feng}
\address{Yan-Quan Feng, School of Mathematics and Statistics, Beijing Jiaotong University, Beijing 100044, China}
\email{yqfeng@bjtu.edu.cn (Y.-Q. Feng)}
%\thanks{$^*$Corresponding author}
% ----------------------------------------------------------------
\begin{abstract}
Let $m$ be a positive integer. A group $G$ is said to be an $m$-DCI-group or an $m$-CI-group if $G$ has the $k$-DCI property or $k$-CI property for all positive integers $k$ at most $m$, respectively. Let $G$ be a dihedral group of order $2n$ with $n\geq 3$. Qu and Yu proved that $G$ is an $m$-DCI-group or $m$-CI-group, for every $m\in \{1,2,3\}$, if and only if $n$ is odd. In this paper, it is shown that $G$ is a $4$-DCI-group if and only if $n$ is odd and not divisible by $9$, and $G$ is a $4$-CI-group if and only if $n$ is odd.
\smallskip


\vskip 5pt

\end{abstract}
\maketitle

\section{Introduction}\label{sect=Int}
%%%%%%%%%%%%%%%%%%%%%%%%%%%%%%%%%%%%%%%%%%%%%%%%%%%%%%%%%%%%%%%%%%%%

In this paper, a (finite) digraph $\Gamma$ is an order pair $(V(\Gamma), Arc(\Gamma))$ of a finite set $V(\Gamma)$ and a set $Arc(\Gamma)$ consisting of ordered pairs of distinct elements from $V(\Gamma)$, while the elements in $V(\Gamma)$ and $Arc(\Gamma)$ are called vertices and arcs of $\Gamma$, respectively. Denote by $\Gamma^-(v)$ and $\Gamma^+(v)$, respectively, the sets of in-neighbors and out-neighbors of a vertex $v$ in a digraph $\Gamma$, that is,
\[
\Gamma^-(v)=\{u\in V(\Gamma)\mid (u,v)\in Arc(\Gamma)\},\, \Gamma^+(v)=\{u\in V(\Gamma)\mid (v,u)\in Arc(\Gamma)\}.
\]
We say a digraph $\Gamma$ is $k$-regular, or of valency $k$, if $|\Gamma^-(v)|=k=|\Gamma^+(v)|$ for all $v\in V(\Gamma)$. For a digraph or group $X$, denote by $\Aut(X)$ its automorphism group. Graphs are also involved in this paper, and it is convenient to treat a graph as a digraph with each edge $\{u,v\}$ equated with two arcs $(u,v)$ and $(v,u)$. Thus we define a graph as a digraph $\Gamma$ in which $(u,v)$ is an arc if and only if so does $(v,u)$; in this case, $\Gamma^-(v)=\Gamma^+(v)$,  written as $\Gamma(v)$.

Let $G$ be a finite group and let $S\subseteq G\setminus \{1\}$, where $1$ is the identity element of $G$. The \emph{Cayley digraph} of $G$ with respect to $S$, denoted by $\Cay(G,S)$, is the digraph having vertex set $G$ such that $(x,y)$ is an arc if and only if $yx^{-1}\in S$. Clearly, $\Cay(G,S)$ is $|S|$-regular. If $S=S^{-1}$, that is, $S$ is closed under taking inverse, then $(x,y)$ is an arc if and only if so does $(y,x)$. In this case, $(x,y)$ is an arc of $\Cay(G,S)$ if and only if so is $(y,x)$, so call $\Cay(G,S)$ a \emph{Cayley graph}.

For subsets $S,\,T\subseteq G\setminus\{1\}$, if $T=S^\sigma$ for some $\sigma\in \Aut(G)$ then $\sigma$ induces an isomorphism from
$\Cay(G,S)$ onto $\Cay(G,T)$, called a \emph{Cayley isomorphism}. A subset $S$ of $G\setminus\{1\}$ is called a \emph{CI-subset} of $G$ if for every $T\subseteq G$ with $\Cay(G,S)\cong \Cay(G,T)$ there exists a Cayley isomorphism between $\Cay(G,S)$ and $\Cay(G,T)$. In this case, $\Cay(G,S)$ is called a \emph{CI-digraph}, or a \emph{CI-graph} when $S=S^{-1}$. For a positive integer $m$, we say that $G$ have the \emph{$m$-DCI property} or \emph{$m$-CI property} if every Cayley digraph or every Cayley graph on $G$ of valency $m$ is a CI-digraph or a CI-graph, respectively. The group $G$ is said to be an \emph{$m$-DCI-group} or an \emph{$m$-CI-group} if $G$ has the $k$-DCI property or $k$-CI property for every positive integer $k\leq m$, respectively; in particular, when $m=|G|-1$, the group $G$ is called a \emph{DCI-group} or a \emph{CI-group}, respectively. Clearly, if $G$ has the $m$-DCI property then $G$ also has the $m$-CI property, and thus, a DCI-group is necessary a CI-group. We see from the definition that $S$ is a CI-subset of $G$ if and  only if so does $G\setminus (S\cup \{1\})$. Then the $m$-DCI property yields the $(|G|-1-m)$-DCI property, and so $G$ is a (D)CI-group if and only if $G$ is an $m$-(D)CI-group with $m=\lfloor{|G|-1\over 2}\rfloor$.

Begun with a conjecture proposed by \'{A}d\'{a}m~\cite{Ad} in 1967, finite $m$-(D)CI-groups have been extensively studied  for over fifty years, referred to surveys in \cite{Alspach-97,C.H.Li8,Palfy}. In current terminology, \'{A}d\'{a}m's conjecture suggests that
every cyclic group $\mathbb{Z}_n$ (with $n>1$) is a DCI-group. This was disproved in 1970 by Elspas and Turner \cite{Elspas}, who proved that $\mathbb{Z}_8$ has no the $3$-DCI property  and $\mathbb{Z}_{16}$ has no the $6$-CI property. However, at that time, they also confirmed the conjecture for a prime $n$. In the following years, the conjecture was proved for certain classes integers $n$, see \cite{Alspach,Babai,Godsil2}. Finally, the finite cyclic DCI-groups were classified by Muzychuk \cite{Mu1,Mu2}, that is, $\mathbb{Z}_n$ is a DCI-group if and only if $n=ab$ with $a\in \{1,2\}$ and $b$ a square-free integer, and $\mathbb{Z}_n$ is a CI-group but not a DCI-group if and only if $n\in\{8,9,18\}$. On the other hand, it has been proved that $\mathbb{Z}_n$ has the $m$-CI property for all positive integers $m\leq \min\{5,n\}$ (see \cite{CHLi,Sun,Toida}). In addition, Li~\cite{C.H.Li1} presented a necessary condition for cyclic groups with the $m$-DCI property, and conjectured that the condition is also sufficient, which has been confirmed by Dobson~\cite{Dobson}.

The first class of nonablian CI-groups, namely the dihedral groups of order twice a prime,  was given by Babai \cite{Babai}, who initiated the study on finite CI-groups other than cyclic groups. Later on, Babai and Frankl \cite{BF-1,BF-2} studied in depth the  CI-groups of odd order and insoluble   CI-groups. Li at al. extended the study of CI-groups and developed the theory of (D)CI-groups, see \cite{CL,C.H.Li7,CI-soluble,LP-2,LP-3,C.H.Li4,C.H.Li5}. It has been shown that a finite $m$-(D)CI-group has strict restrictions on the structure. In particular, the candidates of finite CI-groups and DCI-groups have been reduced to restricted lists, see \cite[Theorem 1.2]{LLP} and \cite[Corollary 1.5]{C.H.Li4}. Despite this, determining which finite groups are DCI-groups or CI-groups is still a highly challenging task. Besides the examples recorded in the references mentioned above, readers are referred \cite{DE1,Feng,KMP,Kov} for more DCI-groups or CI-groups. In this paper, we focus on the dihedral groups, and make an attempt toward to determining the dihedral DCI-groups and CI-groups.

For an integer $n\geq 2$, the dihedral group $\mathrm{D}_{2n}$ is generated by two elements with the following presentation:
\[
\mathrm{D}_{2n}=\langle a,b\mid a^n=b^2=1,\,abab=1 \rangle.
\]
As mentioned above, if $n$ is a prime then $\mathrm{D}_{2n}$ is a DCI-group. Recently, the authors of \cite{XFK} proved that if $\mathrm{D}_{2n}$ has the $m$-DCI property for some $1\leq m\leq n-1$, then $n$ is odd and not divisible by the square of any prime less than $m$, but in general it is unknown whether the converse is true. We also note that by \cite{Qu}, for $m\in \{1,2,3\}$, the dihedral group $\mathrm{D}_{2n}$  is an $m$-DCI-group if and only if $\mathrm{D}_{2n}$ is an $m$-CI-group if and only if $n$ is odd. In this paper, we investigate the isomorphic problem of tetravalent Cayley digraphs on $\mathrm{D}_{2n}$ and consider the $4$-DCI property and $4$-CI property of $\mathrm{D}_{2n}$. Our main result is stated in the following theorem.
%%%%%%%%%%%%%%%%%%%%%%%%%%%%%%%%%%%%%%%%%%%%%%%%%%%%%%%%%%%%%%%%%%%%
\begin{thm}\label{mainth1}
Let $n\ge 3$ be an integer. Then
\begin{enumerate}[{\rm (i)}]
\item $\mathrm{D}_{2n}$ has the $4$-DCI property if and only if $n$ is odd and not divisible by $9$; and
\item $\mathrm{D}_{2n}$ has the $4$-CI property if and only if $n$ is odd.
\end{enumerate}
\end{thm}

By \cite{Qu}, $\mathrm{D}_{2n}$ is a 3-DCI-group if and only if $\mathrm{D}_{2n}$ is a 3-CI-group if and only if $n$ is odd. This together with Theorem \ref{mainth1} gives the following corollary.

\begin{cor}\label{cor1}
Let $n\ge 3$ be an integer. Then
\begin{enumerate}[{\rm (i)}]
\item $\mathrm{D}_{2n}$ is a $4$-DCI-group if and only if $n$ is odd and not divisible by $9$; and
\item $\mathrm{D}_{2n}$ is a $4$-CI-group if and only if $n$ is odd.
\end{enumerate}
\end{cor}

Following this introduction, we state in Section 2 some preliminary results which are used in the following sections. Section 3 proves that every connected Cayley digraphs of $\mathrm{D}_{2n}$ with $n\geq 3$ odd is a CI-digraph, and Section 4 gives a proof of Theorem~\ref{mainth1}.

%%%%%%%%%%%%%%%%%%%%%%%%%%%%%%%%%%%%%%%%%%%%%%%%%%%%%%%%%%%%%%%%%%%%
\vskip 20pt

\section{Preliminaries}\label{sect=Pre}
This section collects some  concepts and  results, which are used in Sections \ref{sect=Tech} and \ref{sect=Prf}.

For a nonempty set $\Omega$, denote by $\mathrm{Sym}(\Omega)$ the \emph{symmetric group} on $\Omega$. In particular, if $\Omega=\{1,\ldots,n\}$ for some positive integer $n$, we write $\mathrm{Sym}(\Omega)$ as $\mathrm{S}_n$ for convenience. Let $\Gamma$ be a digraph. For $B,B_1,B_2\subseteq V(\Gamma)$, denote $[B]$ the subdigraph of $\Gamma$ induced by $B$, and denote $[B_1,B_2]$ the subdigraph of $\Gamma$ obtained from $[B_1\cup B_2]$ by deleting all arcs in $[B_1]$ and $[B_2]$. Denote by $\overrightarrow{\mathrm{K}}_{m,n}$ the orientation of the complete bipartite graph $\mathrm{K}_{m,n}$ with heads of all arcs lying in the  part of size $n$.

Let $\Gamma=\Cay(G,S)$ be a Cayley digraph of a group $G$, and let $A=\Aut(\Gamma)$. Each $g\in G$ induces $R(g)$ of $A$ by the right multiplication on $G$. Write $R(G)=\{R(g)\mid g\in G\}$. Then $R(G)$ is a regular subgroup of $A$, and $g\mapsto R(g)$ gives an isomorphism from $G$ to $R(G)$. Consider the normalizer of $R(G)$ in $\mathrm{Sym}(G)$ on $G$. We have
$\mathrm{N}_{\mathrm{Sym}(G)}(R(G))=R(G)\Aut(G)$, see \cite[Lemma 7.16]{Rotman} for example. Then
$\mathrm{N}_{A}(R(G))=R(G)\Aut(G)\cap A=R(G)\Aut(G,S)$, where $\Aut(G,S)=\{\sigma\in \Aut(G)\mid S^\sigma=S\}$.
In particular, $R(G)$ is a normal subgroup of $A$ if and only if $A=R(G)\Aut(G,S)$; in this case, $\Gamma$ is called a \emph{normal} Cayley digraph of the group $G$. Thus, we have the following result, see also  \cite[Proposition 1.3 and Propositions 1.5]{Xu1}.

\begin{pro}\label{N_AUT}
Let $\Gamma=\Cay(G,S)$ be a Cayley digraph of a group $G$. Then
\begin{enumerate}[{\rm (i)}]
\item $\mathrm{N}_{\Aut(\Gamma)}(R(G))=R(G)\rtimes\Aut(G,S)$; and
\item $\Gamma$ is   normal   with respect to   $G$ if and only if $\Aut(\Gamma)=R(G)\Aut(G,S)$.
\end{enumerate}
\end{pro}
%%%%%%%%%%%%%%%%%%%%%%%%%%%%%%%%%%%%%%%%%%%%%%%%%%%%%%%%%%%%%%%%%%%%
The result stated below is derived from the well-known Babai criterion~\cite{Babai} for determining whether a Cayley digraph is a CI-digraph, also refer to \cite[Theorem 4.1]{C.H.Li8}.

\begin{pro}\label{CI-graph-prop}
Let $\Gamma=\Cay(G,S)$ be a Cayley digraph of a group $G$. Then $\Gamma$ is a CI-digraph if and only if every regular subgroup of $\Aut(\Gamma)$ isomorphic to $G$ is conjugate to $R(G)$ in $\Aut(\Gamma)$. If further $\Gamma$ is normal with respect to the $G$, then $\Gamma$ is a CI-digraph if and only if $R(G)$ is the unique regular subgroup of $\Aut(\Gamma)$ isomorphic to $G$.
\end{pro}
%%%%%%%%%%%%%%%%%%%%%%%%%%%%%%%%%%%%%%%%%%%%%%%%%%%%%%%%%%%%%%%%%%%%
Li~\cite{C.H.Li8} described the abelian Cayley digraphs which are CI-digraphs. By~\cite[Theorem~6.8]{C.H.Li8} and~\cite[Theorem~2]{Delorme}, we have the following proposition.

\begin{pro}\label{4-cyclic}
Let $G$ be a cyclic group of order $n>1$, and let $\Gamma=\Cay(G,S)$ be a connected Cayley digraphs of $G$ with $|S|=4$.
\begin{enumerate}[{\rm (i)}]
\item If $n$ is odd and indivisible by $9$  then $\Gamma$ is a CI-digraph.
\item If $S=S^{-1}$ then $\Gamma$ is a  CI-graph.
\end{enumerate}
\end{pro}
%%%%%%%%%%%%%%%%%%%%%%%%%%%%%%%%%%%%%%%%%%%%%%%%%%%%%%%%%%%%%%%%%%%%
A group $G$ is called an \emph{NDCI-group} or an \emph{NCI-group} if each normal Cayley digraph  or graph  of $G$ is a CI-digraph  or a CI-graph, respectively. The following result, quoted from ~\cite{XFZ}, characterizes the dihedral NDCI-groups and NCI-groups.

\begin{pro}\label{NDCI}
Let $n\geq 2$ be an integer. Then the dihedral group $\mathrm{D}_{2n}$ is an NDCI-group if and only if $\mathrm{D}_{2n}$ is an NCI-group if and only if either $n\in\{2,4\}$ or $n$ is odd.
\end{pro}
%%%%%%%%%%%%%%%%%%%%%%%%%%%%%%%%%%%%%%%%%%%%%%%%%%%%%%%%%%%%%%%%%%%%
According to \cite[Lemma~4.1 and Lemma~4.2]{XFK}, we have the following lemma.

\begin{lem}\label{p-CI}
Let $n>1$ be an odd integer, and let $\Gamma=\Cay(\mathrm{D}_{2n},S)$ be a Cayley digraph.
\begin{enumerate}[{\rm (i)}]
\item If $X$ is a regular subgroup of $\Aut(\Gamma)$ that is isomorphic to $\mathrm{D}_{2n}$, then $X$ is conjugate to $R(\mathrm{D}_{2n})$ in $\Aut(\Gamma)$ if and only if the unique cyclic subgroup of order $n$ of $X$ is conjugate to the unique cyclic subgroup of order $n$ of $R(\mathrm{D}_{2n})$ in $\Aut(\Gamma)$.
\item If the stabilizer of $1$ in $\Aut(\Gamma)$ has order coprime to $n$, then $\Gamma$ is a CI-digraph.
\item If $\Gamma$ is connected and $|S|$ is less than the least prime divisor of $n$, then $\Gamma$ is a CI-digraph.
\end{enumerate}
\end{lem}
%%%%%%%%%%%%%%%%%%%%%%%%%%%%%%%%%%%%%%%%%%%%%%%%%%%%%%%%%%%%%%%%%%%%
The statement below, derived from~\cite[Theorem~1.1]{XFK}, gives a necessary condition of dihedral groups with the $m$-DCI property.

\begin{pro}\label{p-odd-DCI}
Let $n\geq 3$ and $m$ be  integers with $1\leq m\leq n-1$. If $\mathrm{D}_{2n}$ has the $m$-DCI property, then $n$ is odd and indivisible by the square of any prime less than $m$.
\end{pro}
%%%%%%%%%%%%%%%%%%%%%%%%%%%%%%%%%%%%%%%%%%%%%%%%%%%%%%%%%%%%%%%%%%%%

Given $\mathrm{D}_{2n}=\langle a,b\mid a^n=b^2=1,\,bab=a^{-1}\rangle$, it is shown that
\[
\Aut(\mathrm{D}_{2n})=\{\sigma_{r,s}\mid s,r \mbox{ are integers},\,(r,n)=1\},
\]
where $\sigma_{r,s}$ is defined by
\[
(a^i)^{\sigma_{r,s}}=a^{ri},\,(a^jb)^{\sigma_{r,s}}=a^{rj+s}b.
\]
Note that $\sigma_{r,s}=\sigma_{r',s'}$ if and only if $r\equiv r' \pmod n $ and $s\equiv s' \pmod n$. For an integer $w\ge 1$, we have \[
a^{\sigma_{r,s}^w}=a^{r^w},\, b^{\sigma_{r,s}^w}=a^{s(r^{w-1}+\cdots +r+1)}b.
\]
Then
\begin{equation}\label{00}
\sigma_{r,s}^w=1 \text{ if and only if }r^w\equiv 1\pmod n \text{ and }s(r^{w-1}+\cdots +r+1)\equiv 0\pmod n.
\end{equation}

Clearly, if $n$ is even then $a^{n\over 2}$ has order $2$ and lies in the center of $\mathrm{D}_{2n}$; in this case, one can not extend an isomorphism between $\langle a^{n\over 2}\rangle$ and $\langle b\rangle$ to some automorphism of $\mathrm{D}_{2n}$ unless $n=2$. Recall that a  group $G$ is said to be \emph{homogeneous} if every isomorphism between subgroups of  $G$
can be extended to an automorphism of $G$. By \cite[Lemmas~1.6 and 1.9]{Qu}, the next lemma follows.

\begin{lem}\label{homogeneous}
Let $n$ be a positive integer. Then the cyclic group $\mathbb{Z}_n$ is homogeneous, and the dihedral group $\mathrm{D}_{2n}$ is homogeneous if and only if $n=2$ or $n$ is odd.
\end{lem}

\vskip 2pt

We say two subsets $S$ and $T$ of a group $G$ are equivalent,  written as $S\equiv_G T$, if $T=S^\sigma$ for some $\sigma \in \Aut(G)$. Clearly,\[S\equiv_G T\subseteq G\setminus\{1\}\Rightarrow \Cay(G,S)\cong \Cay(G,T).\] In the special case where $G$ is homogeneous, if $S,T\subseteq H\leq G$ then $S\equiv_H T$ if and only if $S\equiv _G T$. Thus we write $S\equiv T$ in place of $S\equiv_H T$ (and $S\equiv_G T$) when $G$ is homogeneous.

\vskip 2pt

By \cite[Theorem~1.1]{Ala} and \cite[Theorem~1]{MP}, if $\Cay(\mathrm{D}_{2n},S)$ is a connected arc-transitive graph of valency $3$ then, up to isomorphism of graphs, $S$ may be chosen as follows:
\begin{enumerate}[{\rm (i)}]
\item  $S=\{b,ab,a^2b\}$ with $n\in\{3,4\}$, or $S=\{b,ab,a^3b\}$ with $n\in \{7,8\}$; in this case, $\Cay(\mathrm{D}_{2n},S)$ is not normal with respect to $\mathrm{D}_{2n}$; or
\item  $S=\{b,ab,a^{r+1}b\}$ with odd $n\geq 13$ and $r^2+r+1\equiv 0\pmod {n}$; in this case, $\Cay(\mathrm{D}_{2n},S)$ is normal with respect to $\mathrm{D}_{2n}$, and $\Aut(\mathrm{D}_{2n},S)\cong \mathbb{Z}_3$.
\end{enumerate}

For the above (ii), we may further require that $0<r<n/2$. In fact, if $r$ is a solution of $x^2+x+1=0$ in $\mathbb{Z}_n$, then
\[
(r^2)^2+r^2+1=(r^2+1)^2-r^2=r^2-r^2=0,
\]
implying that $r^2$ is also a solution. Moreover, it is easy to see that $\Cay(\mathrm{D}_{2n},S^{\sigma_{-r,0}})\cong \Cay(\mathrm{D}_{2n},S)$ and $S^{\sigma_{-r,0}}=\{b,ab,a^{r^2+1}b\}$. Therefore, if $r\geq n/2$, then we can choose $r^2$ in place of $r$, as $r^2=-1-r\in \{1,\ldots,(n-1)/2\}$.

In \cite{Qu}, it was proved that if $n$ is odd then $\mathrm{D}_{2n}$ is a $3$-DCI-group. Thus we have the following result.

\begin{pro}\label{cubic}
Let $n\geq 3$ be an odd integer, and let $\Gamma=\Cay(\mathrm{D}_{2n},S)$ be a connected  arc-transitive Cayley graph with $|S|=3$. Then
\begin{enumerate}[\rm (i)]
\item  $S\equiv S_0:=\{b,ab,a^{r+1}b\}$, where $0<r<{n\over 2}$, $r^2+r+1\equiv 0\pmod {n}$, and either $n\geq 13$ or $n\in \{3,7\}$;
\item $\Aut(\mathrm{D}_{2n},S_0)\geq \langle\sigma_{r,1}\rangle\cong \mathbb{Z}_3$, where the equality holds if and only if $n\ne 3$;
\item $\Gamma$ is normal with respect to $\mathrm{D}_{2n}$ if and only if $n\geq 13$; more precisely, if $n=3$ then $\Gamma\cong {\rm K}_{3,3}$, and $\Aut(\Gamma)\cong (\Sy_3\times \Sy_3){:}\Sy_2$; if $n=7$ then $\Gamma$ is isomorphic to the Heawood graph, and $\Aut(\Gamma)\cong \mathrm{PSL}_3(2).\mathbb{Z}_2$; if $n\ge 13$ then $\Aut(\Gamma)\cong \mathbb{Z}_n{:}\mathbb{Z}_6\cong \mathrm{D}_{2n}{:}\mathbb{Z}_3$.
\end{enumerate}
\end{pro}

\begin{proof}
The conclusion (i) follows from the foregoing argument, and the conclusion (iii) is a corollary of \cite[Theorem~1.1]{Ala} and \cite[Theorem~1]{MP}. Now it remains to show that (ii) holds. Since $\Gamma$ is connected, $\Aut(\mathrm{D}_{2n},S_0)$ acts faithfully on $S_0$, and so $\Aut(\mathrm{D}_{2n},S_0)\lesssim\Sy_3$. It is easy to check that $\sigma_{r,1}\in\Aut(\mathrm{D}_{2n},S_0)$ and $\sigma_{r,1}$ has order $3$. Thus, if $n\geq 11$ then $\Aut(\mathrm{D}_{2n},S_0)=\langle\sigma_{r,1}\rangle$ as $\Gamma$ is normal with respect to $\mathrm{D}_{2n}$. For $n=3$, it is easy to check that $\Aut(\mathrm{D}_{2n},S_0)$ contains $\sigma_{2,0}$, which has order $2$, and so $\Aut(\mathrm{D}_{2n},S_0)=\langle\sigma_{r,1},\sigma_{2,0}\rangle\cong \Sy_3$.

Now let $n=7$. Then $r=2$. By \eqref{00}, each element in $\Aut(\mathrm{D}_{2n})$ of order $2$ has the form of $\sigma_{6,s}$, where $0\leq s\leq 6$. Now $S_0^{\sigma_{6,s}}=\{a^sb, a^{6+s}b, a^{18+s}b\}=\{a^sb,a^{6+s}b,a^{4+s}b, \}\ne \{b,ab,a^3b\}=S_0$. It follows that $\Aut(\mathrm{D}_{2n},S_0)$ contains no element of order $2$. Then $\Aut(\mathrm{D}_{2n},S_0)=\langle\sigma_{r,1}\rangle$. This completes the proof.
\end{proof}

%%%%%%%%%%%%%%%%%%%%%%%%%%%%%%%%%%%%%%%%%%%%%%%%%%%%%%%%%%%%%%%%%%%%
According to \cite[Theorem~1.1]{Kova0} and \cite[Theorem~5.1]{Kova}, we obtain the following proposition about the connected tetravalent arc-transitive Cayley graphs on $\mathrm{D}_{2n}$ with $n$ odd.

\begin{pro}\label{tetravalent-arc-transitive}
Let $n\geq 3$ be an odd integer, and let $\Gamma=\Cay(\mathrm{D}_{2n},S)$ be a connected arc-transitive Cayley graph with $|S|=4$.
If $\Gamma$ is not normal with respect to $\mathrm{D}_{2n}$ then, either $n\in \{5,7,13,15\}$ or $\Gamma\cong \Cay(\mathrm{D}_{2n},\{a,a^{-1},a^2b,b\})$
\end{pro}
%%%%%%%%%%%%%%%%%%%%%%%%%%%%%%%%%%%%%%%%%%%%%%%%%%%%%%%%%%%%%%%%%%%%

Let $\Gamma$ be a digraph, and let $\mathcal{B}$  be a partition of $V(\Gamma)$. The \emph{quotient digraph} $\Gamma_{\mathcal{B}}$ is defined as the digraph with vertex set $\mathcal{B}$ such that $(B_1,B_2)$ is an arc if and only if $(x,y)$ ia an arc of $\Gamma$ for some $x\in B_1$ and $y\in B_2$. For a subgroup $G\leq \Aut(\Gamma)$, if $\mathcal{B}$ is $G$-invariant partition of $V(\Gamma)$, that is, $B^g\in \mathcal{B}$ for all $g\in G$ and $B\in \mathcal{B}$, then $G$ induces an subgroup of $\Aut(\Gamma_{\mathcal{B}})$, say $G^{\mathcal{B}}$. The following fact is well-known and easily proved.

\begin{lem}\label{quotient}
Let $\Gamma$ be a digraph and $G\leq \Aut(\Gamma)$. Assume that $\mathcal{B}$ is a $G$-invariant partition of $V(\Gamma)$, and let $K$ be the kernel of~$G$ acting on $\mathcal{B}$. Then $G^{\mathcal{B}}\cong  G/K$. Moreover, if $G$ is transitively on $V(\Gamma)$ then $G^{\mathcal{B}}$ is transitive on  $\mathcal{B}$,  all $B\in \mathcal{B}$ have equal size, and the stabilizer $G_B$ acts transitively on $B$.
\end{lem}
%%%%%%%%%%%%%%%%%%%%%%%%%%%%%%%%%%%%%%%%%%%%%%%%%%%%%%%%%%%%%%%%%%%%

\vskip 20pt

\section{Connected Cayley digraphs on $\mathrm{D}_{2n}$ of valency $4$}\label{sect=Tech}
%%%%%%%%%%%%%%%%%%%%%%%%%%%%%%%%%%%%%%%%%%%%%%%%%%%%%%%%%%%%%%%%%%%%
In this section, we prove the following result Theorem~\ref{connected-dihedral}, which plays a key role in the proof of Theorem~\ref{mainth1}.

\begin{thm}\label{connected-dihedral}
Let $n\geq 3$ be an odd integer. Then every connected Cayley digraph of valency $4$ on $\mathrm{D}_{2n}$ is a CI-digraph.
\end{thm}
%%%%%%%%%%%%%%%%%%%%%%%%%%%%%%%%%%%%%%%%%%%%%%%%%%%%%%%%%%%%%%%%%%%%%%%%%%%%%%%%%%%%%%%%%%%%%%%%%%%%%%%%%%%%%%%%%%%
First, we construct in the following lemma certain non-normal connected CI-graph of valency $4$ on  the dihedral groups $\mathrm{D}_{2n}=\langle a,b\mid a^n=b^2=1,\,a^b=a^{-1}\rangle$.

Recall that the \emph{lexicographic product} $X[Y]$ of two digraphs is the digraph with vertex set $V(X)\times V(Y)$ such that $(x_1,y_1)$ is adjacent to $(x_2,y_2)$ if and only if either $x_1=x_2$ and $(y_1,y_2)\in Arc(Y)$, or $(x_1,x_2)\in Arc(X)$.

\begin{lem}\label{odd}
Let $n\geq 3$ be an odd integer. Then $\Cay(\mathrm{D}_{2n},\{a,a^{-1},a^2b,b\})$ is a connected CI-graph but not normal with respect to $\mathrm{D}_{2n}$.
\end{lem}

\begin{proof}
Let $S=\{a,a^{-1},a^2b,b\}$ and $\Gamma=\Cay(\mathrm{D}_{2n},S)$. It is easy to see that $S^{-1}=S$ and $\langle S\rangle=\mathrm{D}_{2n}$, and thus $\Gamma$ is a connected Cayley graph. Let $H=\langle ab\rangle$. Then $\mathrm{D}_{2n}=\sum_{i=0}^{n-1}Ha^i$ and $S=Ha \cup Ha^{n-1}$. Note that for every $i\in \{0,\ldots,n-1\}$, we get
\begin{equation}\label{01}
Ha^i=\{a^i,aba^i\} \text{ and } \Gamma(a^i)=\Gamma(aba^i)=Ha^{i-1}\cup Ha^{i+1}.
\end{equation}
Let $\mathcal{B}=\{H,Ha,\ldots,Ha^{n-1}\}$. Noting that $\Gamma$ is a vertex-transitive graph, we can deduce from~\eqref{01} that $\mathcal{B}$ is an $\Aut(\Gamma)$-invariant partition of $V(\Gamma)$. Moreover, we derive from $S=Ha \cup Ha^{n-1}$ that $[Ha^i,Ha^{i+1}]$ of $\Gamma$ by $Ha^i\cup Ha^{i+1}$ is isomorphic to $\mathrm{K}_{2,2}$, which implies that $[Ha^j,Ha^\ell]$ is either a empty graph or isomorphic to $\mathrm{K}_{2,2}$ for two arbitrary elements $j,\ell\in \{0,\ldots,n-1\}$. Thus, it is easily check that $\Gamma\cong \Gamma_{\mathcal{B}}[2\mathrm{K}_1]$,where $2\mathrm{K}_1$ is the empty graph on two vertices; in particular, $\Gamma_{\mathcal{B}}\cong \mathrm{C}_n$, the cycle of length $n$.

Now we show that $\Gamma$ is a CI-graph. Let $\tau_i=(a^i\,aba^i)$ for $0\leq i\leq n-1$, and put $K=\langle \tau_i\mid 0\leq i\leq n-1\rangle$. Then $K\cong \mathbb{Z}_2^n$. Moreover, $K$ is the kernel of $\Aut(\Gamma)$ acting on $\mathcal{B}$, and then we derive from Lemma~\ref{quotient} that
\[
\Aut(\Gamma)/K\cong \Aut(\Gamma)^{\mathcal{B}}\le \Aut(\Gamma_{\mathcal{B}})\cong \mathrm{D}_{2n}.
\]
In particular, $|\Aut(\Gamma)|\le 2n|K|=2^{n+1}n$. Since $n$ is odd and $R(\mathrm{D}_{2n})\cong \mathrm{D}_{2n}$, we conclude that   $R(\mathrm{D}_{2n})$ has no normal subgroup of order a power of $2$. This forces that $K\cap R(\mathrm{D}_{2n})=1$, and so $|\Aut(\Gamma)|\ge |KR(\mathrm{D}_{2n})|=2^n\cdot 2n$. It follows that $\Aut(\Gamma)=KR(\mathrm{D}_{2n})$; in particular, $\Aut(\Gamma)$ is soluble. Then $\langle R(a)\rangle$ is a Hall $2'$-subgroup of $\Aut(\Gamma)$, and every subgroup of $\Aut(\Gamma)$ with order $n$ is conjugate to $\langle R(a)\rangle$. Now Lemma~\ref{p-CI}\,(i) shows that every regular dihedral subgroup of $\Aut(\Gamma)$ is conjugate to $R(\mathrm{D}_{2n})$. Then Proposition~\ref{CI-graph-prop} gives that $\Gamma$ is a CI-graph, as required.

Next we show that $\Gamma$ is not normal with respect to $\mathrm{D}_{2n}$. Take $\tau=(1\,ab)\in \mathrm{Sym}(\mathrm{D}_{2n})$. We drive from \eqref{01} that $\Gamma(1)=S=\Gamma(ab)$, and so the transposition $\tau$ is an automorphism of $\Gamma$. Now, $(1,a)^{\tau^{-1}R(a)\tau}=(b,a^2)$. This implies that $\tau^{-1}R(a)\tau\notin R(\mathrm{D}_{2n})$, and so $R(\mathrm{D}_{2n})$ is not a normal subgroup of $\Aut(\Gamma)$, and the lemma follows.
\end{proof}
%%%%%%%%%%%%%%%%%%%%%%%%%%%%%%%%%%%%%%%%%%%%%%%%%%%%%%%%%%%%%%%%%%%%%%%%%%%%%%%%%%%%%%%%%%%%%%%%%%%%%%%%%%%%%%%%%%%

Recall that for an odd prime $p$, both $\mathrm{D}_{2p}$ and $\mathrm{D}_{6p}$ are CI-groups, refer to \cite{Babai,DE1}. Applying Lemma~\ref{odd} and other facts, we have the following result, which says that Theorem~\ref{connected-dihedral} holds for arc-transitive digraphs.

\begin{lem}\label{arctran}
Let $n\geq 3$ be an odd integer, and let $\Gamma=\Cay(\mathrm{D}_{2n},S)$ be of valency $4$. If $\Gamma$ is connected and arc-transitive, then $\Gamma$ is a CI-digraph.
\end{lem}

\begin{proof}
Assume that $\Gamma$ is  connected, i.e., $\mathrm{D}_{2n}=\langle S\rangle$. Then $S$ contains at least one element $u$ of order $2$, and thus both $(1,u)$ and $(u,1)$ are arcs of $\Gamma$. Since $\Gamma$ is arc-transitive, it follows that $\Gamma$ is a graph. If $\Gamma$ is normal with respect to $\mathrm{D}_{2n}$, then Proposition~\ref{NDCI} asserts that $\Gamma$ is a CI-graph, as required. Thus, we suppose next that  $\mathrm{D}_{2n}$ is not normal in $\Aut(\Gamma)$.

By Proposition~\ref{tetravalent-arc-transitive}, either $n\in \{5,7,13,15\}$ or $\Gamma\cong \Cay(\mathrm{D}_{2n},\{a,a^{-1},a^2b,b\})$. The latter case implies that $\Gamma$ is a CI-graph, see Lemma~\ref{odd}. Now let $n\in \{5,7,13,15\}$; in particular, $2n$ has the form of $2p$ or $6p$ for some prime $p$. In this case, $\mathrm{D}_{2n}$ is a CI-group, refer to \cite{Babai,DE1}. Then $\Gamma$ is a CI-graph. This completes the proof.
\end{proof}
%%%%%%%%%%%%%%%%%%%%%%%%%%%%%%%%%%%%%%%%%%%%%%%%%%%%%%%%%%%%%%%%%%%%%%%%%%%%%%%%%%%%%%%%%%%%%%%%%%%%%%%%%%%%%%%%%%%

In the following, we deal with those Cayley digraphs which are not arc-transitive.
%%%%%%%%%%%%%%%%%%%%%%%%%%%%%%%%%%%%%%%%%%%%%%%%%%%%%%%%%%%%%%%%%%%%%%%%%%%%%%%%%%%%%%%%%%%%%%%%%%%%%%%%%%%%%%%%%%%

\begin{lem}\label{non-arctran-dihedral}
Let $n\geq 3$ be an odd integer, and let $\Gamma=\Cay(\mathrm{D}_{2n},S)$ be  a connected Cayley digraph of valency $4$. Let $A=\Aut(\Gamma)$ and $A_1$ be the stabilizer of $1$. Suppose that $A_1$ has an orbit $S_0$ on $S$ with $|S_0|=3$ and $\langle S_0\rangle$ dihedral. Then $\Gamma$ is a CI-digraph.
\end{lem}
\begin{proof}
In view of Proposition~\ref{NDCI}, we may assume  that $R(\mathrm{D}_{2n})$ is not normal in $A$. Let $\Gamma_0=\Cay(\mathrm{D}_{2n},S_0)$. Then $A\leq \Aut(\Gamma_0)$, and so $R(\mathrm{D}_{2n})$ is not normal in $\Aut(\Gamma_0)$. In addition, $\Gamma_0$ is arc-transitive. Since $\langle S_0\rangle$ is a dihedral group, $S_0$ contains at least one involution. Then $\Gamma_0$ is an arc-transitive graph.

Let $H=\langle S_0\rangle$. If $\Gamma_0$ is connected then, since $R(\mathrm{D}_{2n})\ntrianglelefteq\Aut(\Gamma_0)$ and $n\geq 3$ is odd, Proposition~\ref{cubic} asserts that $n=3$ or $7$; in this case, $\mathrm{D}_{2n}$ is a DCI-group, and so $\Gamma$ is a CI-digraph. Thus we suppose next that $H\ne \mathrm{D}_{2n}$. Then $H\cong \mathrm{D}_{2k}$, where $k$ is a proper divisor of $n$. We have $H=\langle a^{m},b\rangle$, where $m={n\over k}\ge 3$.

Since $n$ is odd, by Lemma \ref{homogeneous},  $\mathrm{D}_{2n}$ is homogeneous. Noting  that $\Cay(H,S_0)$ is a connected arc-transitive cubic graph, it follows Proposition \ref{cubic} that there exists $\alpha\in \Aut(\mathrm{D}_{2n})$ such that $S_0^\alpha= \{b,a^mb,a^{m(r_0+1)}b\}$, where $0<r_0<{k\over 2}$, $r_0^2+r_0+1\equiv 0\pmod {k}$, and either $k\geq 13$ or $k\in \{3,7\}$. Note that $S$ is a CI-subset if and only if so is $S^\alpha$. In the following, without loss of generality, we assume that
\[
S_0= \{b,a^mb,a^{m(r_0+1)}b\}, S= \{b,a^mb,a^{m(r_0+1)}b,a^s\} \mbox{ or }\{b,a^mb,a^{m(r_0+1)}b,a^sb\}
\]
for some integer $1\leq s\leq n-1$. Since $\Gamma$ is connected, we have $\langle S\rangle =\mathrm{D}_{2n}$, and then $\langle a^m, a^s\rangle=\langle a\rangle$, which forces that $\gcd(m,s)=1$.  Recall that $k\geq 13$ or $k\in \{3,7\}$.

\medskip
\noindent{\bf Case 1:}. $k=3$.

In this case, $r_0=1$ and so $S_0=\{b,a^mb,a^{2m}b\}$.
Let $T\subseteq \mathrm{D}_{2n}$ with $\Gamma\cong \Cay(\mathrm{D}_{2n},T)$. Then $\Cay(\mathrm{D}_{2n}, T)$ has an arc-transitive subgraph $\Cay(\mathrm{D}_{2n}, T_0)$, which is isomorphic to $\Gamma_0=\Cay(\mathrm{D}_{2n}, S_0)$. Since $\mathrm{D}_{2n}$ is a $3$-DCI-group (see \cite{Qu}), there exists $\beta\in \Aut(\mathrm{D}_{2n})$ such that $T_0^\beta=S_0$ and
\[
Y:= T^\beta=\{b,a^mb,a^{2m}b,a^t\}\mbox{ or } \{b,a^mb,a^{2m}b,a^tb\},
\]
where $1\leq t\leq n-1$ with $\gcd(t,m)=1$. Noting  that $\Gamma\cong \Cay(\mathrm{D}_{2n},Y)$, it follows that $\Cay(\mathrm{D}_{2n},S\setminus S_0)\cong \Cay(\mathrm{D}_{2n},Y\setminus S_0)$. This implies that $Y\setminus S_0=\{a^t\}$  if $S\setminus S_0=\{a^s\}$, and $Y\setminus S_0=\{a^tb\}$ otherwise.

Suppose first that $S\setminus S_0=\{a^s\}$. Then $Y\setminus S_0=\{a^t\}$. Noting that $\Cay(\mathrm{D}_{2n},\{a^s\})\cong \Cay(\mathrm{D}_{2n},\{a^t\})$, it follows that $a^s$ and $a^t$ has the same order, and so $\gcd(s,n)=\gcd(t,n)$. Recalling that $\gcd(s,m)=\gcd(t,m)=1$ and $m={n\over k}={n\over 3}$, we derive that
\[
\text{ either }\gcd(s,n)=1=\gcd(t,n), \text{ or }\gcd(s,n)=3=\gcd(t,n) \text{ and }\gcd(3,m)=1.
\]

Put $s=3^i\cdot s_1$ and $t=3^i\cdot t_1$ for $i\in \{0,1\}$. Then $\gcd(s_1,n)=1=\gcd(t_1,n)$. This implies that $xs_1\equiv t_1\pmod n$ for some integer $x$ with $\gcd(x,n)=1$. It follows that $S_0^{\sigma_{x,0}}=\{b, a^{xm}b, a^{2xm}b\}$, and $(a^s)^{\sigma_{x,0}}=a^t$. Since $\gcd(3,x)=1$, we have $x\equiv\pm 1\pmod 3$. Then $S_0^{\sigma_{x,0}}=\{b, a^{m}b, a^{2m}b\}=S_0$, and $S^{\sigma_{x,0}}=Y$, which implies that $S\equiv Y\equiv T$. Therefore, $S$ is a CI-subset, that is, $\Gamma$ is a CI-digraph.

Now let $S\setminus S_0=\{a^sb\}$. We have  $Y\setminus S_0=\{a^tb\}$. Note that either $\gcd(s,n)=3^i=\gcd(t,n)$ for some $i\in \{0,1\}$, or exactly one of $s$ and $t$ is coprime to $n$. For the former case, choosing $x$ as above, we have  $S^{\sigma_{x,0}}=Y$, and so $S\equiv Y\equiv T$. Thus assume the latter case occurs, with out loss of generality, we let $\gcd(s,n)=1$ and $\gcd(t,n)=3$. Recalling that $\gcd(t,m)=1$, we have $\gcd(3,m)=1$. Since $\gcd(3,m)=1=\gcd(3,s)$, we have $s\equiv \pm 1\pmod 3$ and $m\equiv \pm 1\pmod 3$. Choose $y\in \{-1,1\}$ such that $s':=s+ym$ is divisible by $3$. Then $\gcd(s',m)=1$ and $\gcd(s',n)=3=\gcd(t,n)$. By a similar argument as above, there exists an integer $x'$ such that
$\{b,a^mb,a^{2m}b, a^{s'}b\}^{\sigma_{x',0}}=Y$. Note that
\[
S^{\sigma_{1,ym}}=\{a^{my}b, a^{m(y+1)}b, a^{m(y+2)}b, a^{s'}b\}=\{b,a^mb,a^{2m}b, a^{s'}b\}.
\]
It follows that $S\equiv \{b,a^mb,a^{2m}b, a^{s'}b\}\equiv Y\equiv T$, and so $S\equiv T$. Therefore, $S$ is a CI-subset, that is, $\Gamma$ is a CI-digraph.

\medskip
\noindent {\bf Case 2:} $k=7$ or $k\ge 13$.

Note that $\Gamma_0=\Cay(\mathrm{D}_{2n},S_0)$ has exactly $m$ connected components, say $\Lambda_i:=[\langle a^m,b\rangle a^{is}]$, $0\leq i\leq m-1$, of which each is isomorphic to the arc-transitive graph $\Cay(H,S_0)$.
Moreover, $A$ permutes these components. Put $\mathcal{B}=\{\langle a^m,b\rangle a^{is}\mid 0\leq i\leq m-1\}$, and let $K$ be the kernel of $A$ acting on $\mathcal{B}$. It follows that $R(\langle a^m\rangle)\leq K$.

\medskip
\noindent{\bf Claim I:} The quotient graph $\Gamma_{\mathcal{B}}$ is a cycle of length $m$, $A/K\cong \mathrm{D}_{2m}$, $A=K\rtimes R(\langle a^s,b\rangle)$, and $K$ is edge-transitive but not vertex-transitive on each $\Lambda_i$.

Note the $\Lambda_i$ is a connected bipartite graph with the bipartition $(\langle a^m\rangle a^{is}, \langle a^m\rangle b a^{is})$. Assume first that $S\setminus S_0=\{a^sb\}$. We have
\[
a^sb\langle a^m\rangle a^{is}=\langle a^m\rangle b a^{is-s},\,a^sb\langle a^m\rangle ba^{is}=\langle a^m\rangle a^{is+s}, \forall\, i\in \{0,\ldots,m-1\}.
\]
It follows that $\{\langle a^m,b\rangle a^{is},\langle a^m,b\rangle a^{js}\}$ is an edge of $\Gamma_{\mathcal{B}}$ if and only if $j-i\equiv\pm 1\pmod m$, and $[\langle a^m,b\rangle a^{is},\langle a^m,b\rangle a^{(i+1)s}]$ is the union of a perfect matching $[\langle a^m\rangle ba^{is},\langle a^m\rangle a^{(i+1)s}]~(\cong k\mathrm{K}_{1,1})$ and an independent set $\langle a^m\rangle a^{is}\cup \langle a^m\rangle ba^{(i+1)s}$. Thus $\Gamma_{\mathcal{B}}\cong \mathrm{C}_m$, the cycle of length $m$, and $K$ fixes each of $\langle a^m\rangle ba^{is}$ and $\langle a^m\rangle a^{is}$ set-wise.

Next assume that $S\setminus S_0=\{a^s\}$. We have
\[
a^s\langle a^m\rangle a^{is}=\langle a^m\rangle a^{is+s}, \,a^s\langle a^m\rangle ba^{is}=\langle a^m\rangle b a^{is-s}, \forall\, i\in \{0,\ldots,m-1\}.
\]
It follows that $Arc([\langle a^m,b\rangle a^{is},\langle a^m,b\rangle a^{js}])\ne \emptyset$ if and only if $j-i\equiv\pm 1\pmod m$, and the subdigraph $[\langle a^m,b\rangle a^{is},\langle a^m,b\rangle a^{(i+1)s}]$ is the union of two directed matching $[\langle a^m\rangle a^{is}, \langle a^m\rangle a^{(i+1)s}]~(\cong k\overrightarrow{\mathrm{K}}_{1,1})$ and $[\langle a^m\rangle ba^{(i+1)s}, \langle a^m\rangle ba^{is}]~(\cong k\overrightarrow{\mathrm{K}}_{1,1})$, both of which do not contain any isolated vertex.
This implies that $\Gamma_{\mathcal{B}}$ is also isomorphic to $\mathrm{C}_m$, and $K$ fixes each of $\langle a^m\rangle a^{is}$ and $\langle a^m\rangle ba^{is}$ set-wise.

The argument above shows that $A/K\lesssim\Aut(\mathrm{C}_m)\cong \mathrm{D}_{2m}$, and $K$ preserves the bipartition of each $\Lambda_i$.
Further, it is easily shown that $R(\langle a^s,b\rangle)$ acts transitively but not regularly on the vertex set of $\Gamma_{\mathcal{B}}$. Then
\[
A/K\cong \mathrm{D}_{2m} \text{ and } A=K\rtimes R(\langle a^s,b\rangle),
\]
desired as in  the claim.

For each $i$, since $|A:A_{\langle a^m,b\rangle a^{is}}|=m$, we conclude that $K$ has index $2$ in $A_{\langle a^m,b\rangle a^{is}}$. Since $\Lambda_i$ is a connected component of $\Gamma_0$ and $\Gamma_0$ is arc-transitive, $A_{\langle a^m,b\rangle a^{is}}$ acts transitively on $Arc(\Lambda_i)$, and hence $K$ acts transitively on the edge set of $\Lambda_i$. Suppose that $K$ is unfaithful on one of $\langle a^m\rangle a^{is}$ and $\langle a^m\rangle b a^{is}$. Noting that $\Lambda_i$ has valency $3$, it follows that $\Lambda_i$ is the complete bipartite graph of order $6$. This implies that $a^m$ has order $3$, and so $k=3$, which is not the case. Then the claim follows.

\medskip
\noindent{\bf Claim II:} If $K\cong \mathbb{Z}_k{:}\mathbb{Z}_3$, then $R(\mathrm{D}_{2n})\unlhd A$.

Assume that $K\cong \mathbb{Z}_k{:}\mathbb{Z}_3$. Then $R(\langle a^m\rangle)\unlhd K$ and $|A|=2m|K|=3|R(\mathrm{D}_{2n})|$. In particular, $A=R(\mathrm{D}_{2n})K_1$. It further follows from $R(\langle a^m\rangle)\unlhd R(\mathrm{D}_{2n})$ that $R(\langle a^m\rangle)\unlhd A$. Let $C$ be the centralizer of $R(\langle a^m\rangle)$ in $A$. Clearly, $R(\langle a\rangle)\leq C\unlhd A$, and so
\[
A=R(\mathrm{D}_{2n})K_1=C(K\langle R(b)\rangle)=(CK_1)\langle R(b)\rangle.
\]
Noting that $K$ is not abelian, we have $K\not\leq C$; in particular, $K_1\not\leq C$, and so $C\cap K_1=1$. Then $CK_1/C\cong K_1\cong \mathbb{Z}_3$. In addition, $R(b)\not\in C$, and so $C\langle R(b)\rangle/C\cong \mathbb{Z}_2$. Considering the action of $A$ on $R(\langle a^m\rangle)$ by conjugation, we have
\[
A/C\lesssim\Aut(R(\langle a^m\rangle))\cong \Aut(\mathbb{Z}_k).
\]
Then $A/C$ is abelian, and so
\[
A/C=(CK_1/C)(C\langle R(b)\rangle /C)\cong \mathbb{Z}_3\times \mathbb{Z}_2.
\]
This implies that $C=R(\langle a\rangle)$, and $R(\mathrm{D}_{2n})/C\unlhd A/C$. Then
$R(\mathrm{D}_{2n})\unlhd A$, as claimed.

\medskip

Now we are ready to show that $S$ is a CI-subset. Note that $k\geq 13$ or $k\in \{3,7\}$. For  $k\ge 13$, by Claim~I and (iii) of Proposition \ref{cubic}, $K\cong \mathbb{Z}_k{:}\mathbb{Z}_3$; in this case, by Claim~II,  $R(\mathrm{D}_{2n})\unlhd A$, and so $S$ is a CI-subset by Proposition \ref{NDCI}. Thus, we let $k=7$. If $K\cong \mathbb{Z}_7{:}\mathbb{Z}_3$ then $S$ is a CI-subset by Claim~II and Proposition \ref{NDCI}. Thus assume next that $K\not\cong \mathbb{Z}_7{:}\mathbb{Z}_3$.

By (iii) of Proposition \ref{cubic}, $\Aut(\Lambda_0)\cong \mathrm{PSL}_3(2){.}\mathbb{Z}_2$. Recall that $K$ is an edge-transitive but not vertex-transitive subgroup of $\Aut(\Lambda_0)$, see Claim~I. We have $K\lesssim \mathrm{PSL}_3(2)$. Checking the subgroups of $\mathrm{PSL}_3(2)$ in the Atlas \cite{Atlas}, since $|K|$ is divisible by $7$, we conclude that $K\cong \mathrm{PSL}_3(2)$. Let $C$ be the centralizer of $K$ in $A$. Then
\[
A/KC\cong (A/C)/(KC/C)\lesssim\Aut(K)/\mathrm{Inn}(K)\cong \mathbb{Z}_2.
\]
In particular, $|A:(KC)|\leq 2$. Since $R(a)$ has odd order $n$, we conclude that $R(a)\in KC$. In addition, since $|(A/K):(KC/K)|=|A:(KC)|\leq 2$ and $A/K\cong \mathrm{D}_{2m}$, we have
\[
KC/K\cong \mathbb{Z}_m \text{ or }\mathrm{D}_{2m}.
\]
Since $K$ is a nonabelian simple group, we have $K\cap C=1$, and so
\[
KC=K\times C \text{ ,and }C\cong KC/K\cong \mathbb{Z}_m \text{ or }\mathrm{D}_{2m}.
\]
It follows from $|A:(KC)|\leq 2$ and $K\cong \mathrm{PSL}_3(2)$ that either $\mathrm{PSL}_3(2)\times \mathbb{Z}_m$ or $\mathrm{PSL}_3(2)\times \mathrm{D}_{2m}$ contains a cyclic subgroup of order $7m$. The only possibility is that $(7,m)=1$.

Now let $R$ be an arbitrary cyclic subgroup of order $n$ in  $A$. Then, since $n=7m$ and $(7,m)=1$, we have $R=C_1\times C_2$ with $C_1\cong \mathbb{Z}_7$ and $C_2\cong \mathbb{Z}_m$. Note that $|KC_1:K|={|C_1|\over |K\cap C_1|}=1$ or $7$. Since $KC_1\leq A$, we have
\[
2m=|A:K|=|A:(KC_1)||KC_1:K|.
\]
This implies that $|KC_1:K|=1$, and so $C_1\leq K$. Then we can obtain
\[
R\leq \mathbf{C}_{KC}(C_1)=\mathbf{C}_{K}(C_1)\times \mathbf{C}_{C}(C_1)=C_1\times C,
\]
which implies that $C_2\leq C$. Clearly, $C$ has a unique cyclic subgroup of order $m$, and all cyclic subgroups of order $7$ in $K$ are conjugate under $K$. It follows that all cyclic subgroups of order $7m$ in $KC$ are conjugate under $KC$. Then we derive from $|A:(KC)|\leq 2$ that all cyclic subgroups of order $n$ in $A$ are contained and conjugate in $KC$. Then, by Proposition~\ref{CI-graph-prop} and (i) of Lemma \ref{p-CI}, $S$ is a CI-subset and hence $\Gamma$ is a CI-digraph. This completes the proof.
\end{proof}
%%%%%%%%%%%%%%%%%%%%%%%%%%%%%%%%%%%%%%%%%%%%%%%%%%%%%%%%%%%%%%%%%%%%%%%%%%%%%%%%%%%%%%%%%%%%%%%%%%%%%%%%%%%%%%%%%%%

In the following, we can deduce from Lemma~\ref{non-arctran-dihedral} and other facts that Theorem~\ref{connected-dihedral} holds for non-arc-transitive digraphs.

\begin{lem}\rm\label{non-arctran}
Let $n\geq 3$ be an odd integer, and let $\Gamma=\Cay(\mathrm{D}_{2n},S)$ be a connected Cayley digraph with $|S|=4$. If $\Gamma$ is not arc-transitive, then $\Gamma$ is a CI-digraph.
\end{lem}
\begin{proof}
Let $A=\Aut(\Gamma)$, and so $A$ does not acts transitively on $Arc(\Gamma)$. Let $A_1$ be the stabilizer of $1$ in $A$. If $\gcd(|A_1|,n)=1$ then, by (ii) of Lemma~\ref{p-CI}, $\Gamma$ is a CI-digraph. Thus, we may assume that $\gcd(|A_1|,n)\neq 1$. Then $3$ divides $|A_1|$ as $n$ is odd. Let $\alpha$ be an element of order $3$ in $A_1$. By \cite[Lemma 2.6.1]{GR}, $\Gamma$ is strongly connected, and hence there exist an $m$-arc $(1=u_0,u_1,\ldots,u_m)$ such that $\alpha$ fixes $u_i$ for every $0\leq i\leq m-1$ and has a $3$-orbit on $\Gamma^+(u_m)$. Since $\Gamma$ is not arc-transitive, $A_{u_m}$ has two orbits on $\Gamma^+(u_m)$ with length $1$ and $3$, and since $A$ is vertex-transitive on $\Gamma$, $A_1$ has exactly two orbits on $S$ with length $1$ and $3$.

Let $S_0$ be the $A_1$-orbit of length $3$. Then $A$ acts transitively on $Arc(\Cay(\mathrm{D}_{2n},S_0))$. If $\langle S_0\rangle$ is dihedral, then $\Gamma$ is a CI-digraph by Lemma~\ref{non-arctran-dihedral}. Thus, we may assume that $\langle S_0\rangle$ is not dihedral, and since $\Gamma$ is  connected, we may write
\[
S_0=\{a^s,a^t,a^r\} \mbox{ and }S=\{a^s,a^t,a^r,a^ib\}.
\]
Let $T\subseteq \mathrm{D}_{2n}$ with $\Gamma\cong \Cay(\mathrm{D}_{2n},T)$. Since $A_1$ acting on $S$ has an orbit $S_0$ of length $3$ such that $\langle S_0\rangle\leq \langle a\rangle$, it is easy to check that $|T\cap\langle a\rangle|=3$. Similarly, we have
\[
T_0:=\{a^{s'},a^{t'},a^{r'}\} \mbox{ and } T=\{a^{s'},a^{t'},a^{r'},a^jb\},
\]
where $\Cay(\mathrm{D}_{2n},S_0)\cong \Cay(\mathrm{D}_{2n},T_0)$.  Since $\mathrm{D}_{2n}$ is a $3$-DCI group, $\mathrm{D}_{2n}$ has an automorphism $\beta$ such that $S_0^\beta=T_0$. It follows that $S^\beta=\{a^{s'},a^{t'},a^{r'},a^kb\}$, and  since $\sigma_{1,j-k}$ fixes $a$ and maps $a^kb$ to $a^jb$, we have $S\equiv T$ and hence $\Gamma$ is a CI-digraph, as required. This completes the proof.
\end{proof}

Finally, Theorem \ref{connected-dihedral} follows from Lemma \ref{arctran} and \ref{non-arctran}.

\vskip 20pt

\section{Proof of Theorem~\ref{mainth1}}\label{sect=Prf}

For positive integer $k$ and a vertex in a graph $\Gamma$, denote by $\Gamma_k(u)$ the set of vertices at distance $k$ from $u$. Note that $\Gamma_1(u)=\Gamma(u)$, called the \emph{neighborhood} of $u$ in $\Gamma$.

\begin{lem}\label{even}
Let $n\geq 4$ be an even  integer. Let $S=\{a,a^{-1},ab,a^{{n\over 2}+1}b\}\subseteq \mathrm{D}_{2n}$, and $\Gamma=\Cay(\mathrm{D}_{2n},S)$. Then $S$ is not a CI-subset, and if $n>4$ then $R(\mathrm{D}_{2n})\unlhd \Aut(\Gamma)$.
\end{lem}
\begin{proof}
Note that $S^{-1}=S$ and $\langle S\rangle=\mathrm{D}_{2n}$. Then $\Gamma$ is a connected Cayley graph. Let $A=\Aut(\Gamma)$. Assume that $n=4$. Then $\Gamma$ is isomorphic to the complete bipartite graph of order $8$, and so $A\cong (\Sy_4\times\Sy_4){:}\Sy_2$. Computation with GAP \cite{GAP} shows that $A$ has two regular subgroups which are isomorphic to $\mathrm{D}_{8}$ but not conjugate in $A$. Then $S$ is not a CI-subset by Proposition \ref{CI-graph-prop}, and the lemma holds for $n=4$.

Assume that $n=2m>4$ in the following. Let $A_1$ be the stabilizer $1$ in $A$. Then
\begin{equation}\label{2-2}
\mathbb{Z}_2^2\cong\langle \sigma_{1,m},\sigma_{-1,2} \rangle\leq\Aut(\mathrm{D}_{2n},S)\leq A_1.
\end{equation}
Clearly, $\Gamma_2(1)=\{a^2,b,a^{m}b,a^{m},a^2b,a^{m+2}b,a^{-2}\}$. Since $n=2m>4$, we have $|\Gamma_2(1)|=7$. The induced subgraph $[\Gamma_0(1)\cup \Gamma_1(1)\cup \Gamma_2(1)]$ can be depicted as Figure~1.

\begin{figure}[htb]
\begin{center}
\begin{picture}(260,100)(-10,40)
\thicklines

\put(120,140){\circle*{5}}
\put(0,100){\circle*{5}}
\put(60,100){\circle*{5}}
\put(180,100){\circle*{5}}
\put(240,100){\circle*{5}}

\put(0,70){\circle*{5}}
\put(30,70){\circle*{5}}
\put(90,70){\circle*{5}}
\put(120,70){\circle*{5}}
\put(150,70){\circle*{5}}
\put(210,70){\circle*{5}}
\put(240,70){\circle*{5}}

\qbezier(120,140)(0,120)(0,100)
\qbezier(120,140)(60,120)(60,100)
\qbezier(120,140)(180,120)(180,100)
\qbezier(120,140)(240,120)(240,100)

\put(0,100){\line(0,-1){30}}
\put(0,100){\line(1,-1){30}}
\put(0,100){\line(3,-1){90}}
\put(60,100){\line(-1,-1){30}}
\put(60,100){\line(3,-1){90}}
\put(60,100){\line(2,-1){60}}
\put(180,100){\line(1,-1){30}}
\put(180,100){\line(-3,-1){90}}
\put(180,100){\line(-2,-1){60}}
\put(240,100){\line(-3,-1){90}}
\put(240,100){\line(-1,-1){30}}
\put(240,100){\line(0,-1){30}}

\put(117,143){${\small 1}$}
\put(-10,99){${\small a}$}
\put(44,99){${\small ab}$}
\put(144,99){${\small a^{m+1}b}$}
\put(243,99){${\small a^{-1}}$}

\put(-5,56){${\small a^2}$}
\put(27,56){${\small b}$}
\put(81,56){${\small a^{m}b}$}
\put(113,56){${\small a^{m}}$}
\put(140,56){${\small a^2b}$}
\put(194,56){${\small a^{m+2}b}$}
\put(233,56){${\small a^{-2}}$}
\put(-20,30){Figure 1. The induced subgraph $[\Gamma_0(1)\cup \Gamma_1(1)\cup \Gamma_2(1)]$ in $\Gamma$}
\end{picture}
\end{center}
\end{figure}

First, we prove that $R(\mathrm{D}_{2n})\unlhd \Aut(\Gamma)$. Let $A^*_1$ be the kernel of $A_1$ acting on $\Gamma(1)=S$. For $\Gamma_2(1)$, according to Figure~$1$, $A_1^*$ fixes $\{a^2,a^{-2}\}$ pointwise, and $A_1^*$ fixes
\[
\{b,a^{m}b\}, \{b,a^{m},a^2b\}, \{a^{m}b,a^{m},a^{m+2}b\} \text{ and }\{a^2b,a^{m+2}b\}
\]
setwise, respectively. Of course, $A_1^*$ fixes those sets obtained from the last four sets by set operations: $\cup$, $\cap$ and $\setminus$. It follows that $A_1^*$ fixes $\Gamma_2(1)$ pointwise. By the transitivity of $A$ on $V(\Gamma)$, $A_w^*$ fixes $\Gamma_2(w)$ pointwise for every $w\in V(\Gamma)$. Since $\Gamma$ is connected,  for each positive integer $k$, an easy inductive argument on $k$ gives rise to that $A_1^*$ fixes $\Gamma_k(1)$ pointwise. Thus, $A_1^*$ fixes each vertex of $\Gamma$, and so $A_1^*=1$. Then $A_1$ acts faithfully on $S$. Again by Figure~1, $A_1$ fixes $\{a,a^{-1}\}$ and $\{ab,a^{m+1}b\}$ setwise, respectively. This implies that $|A_1|\leq 4$. It follows from (\ref{2-2}) that $A_1=\langle \sigma_{1,m},\sigma_{-1,2} \rangle=\Aut(\mathrm{D}_{2n},S)$. Then, by Proposition~\ref{N_AUT}, $R(\mathrm{D}_{2n})\unlhd \Aut(\Gamma)$.

\vskip 2pt

Next, we prove that $S$ is not a CI-subset. According to Proposition~\ref{CI-graph-prop}, it suffices to shown that $A$ has a regular dihedral subgroup, which is not $R(\mathrm{D}_{2n})$. Write $\beta=\sigma_{-1,2}$. Then $\beta R(a^2b)\beta=R(b)$ and $\beta^2=1$. We get $(R(a^2b)\beta)^2=R(a^2)$, which has order $m$. Thus either $R(a^2b)\beta$ has order $m$ and $m$ is odd, or $R(a^2b)\beta$ has order $n$. For the former case, we have
\[
1=(R(a^2b)\beta)^{m}=((R(a^2b)\beta)^2)^{m-1\over 2}(R(a^2b)\beta)=R(a^2)^{m-1\over 2}(R(a^2b)\beta)
=R(a^{m+1}b)\beta\neq 1,
\]
a contradiction. Therefore, $R(a^2b)\beta$ has order $n$. Noting that $\beta R(ab){\beta}=R(ab)$, we have
\[
R(ab)R(a^2b)\beta R(ab)= R(ab)R(a^2b)R(ab)\beta=R(b)\beta.
\]
Calculation shows that $\beta R(a)\beta =R(a^{-1})$, and so
\[
(R(a^2b)\beta)^{-1}=\beta R(a^2b)=\beta R(a)R(ab)\beta^2=\beta R(a)\beta R(ab)\beta=R(a^{-1}) R(ab)\beta=R(b)\beta.
\]
It follows that  $R(ab)(R(a^2b)\beta){R(ab)}=(R(a^2b)\beta)^{-1}$. Since $R(ab)$ has order $2$, we have
\[
L:=\langle R(a^2b)\beta, R(ab)\rangle\cong \mathrm{D}_{2n}.
\]
If $L=R(\mathrm{D}_{2n})$, then $\beta\in R(\mathrm{D}_{2n})$, a contradiction. This implies that $L\ne R(\mathrm{D}_{2n})$. Noting that  $\langle a^2,ab\rangle$ is a dihedral group of order $n=2m$, we have $|R(\mathrm{D}_{2n}):\langle R(a^2), R(ab)\rangle|=2$. Since $R(a^2)=(R(a^2b)\beta)^2\in L$ and $R(ab)\in L$, we have $L\cap R(\mathrm{D}_{2n})=\langle R(a^2), R(ab)\rangle$. Clearly, $\langle R(a^2), R(ab)\rangle$ has two orbits on $\mathrm{D}_{2n}$, that is, $\langle a^2,ab\rangle$ and $b\langle a^2,ab\rangle$. Noting that $1^{R(a^2b)\beta}=(a^2b)^\beta=b\in b\langle a^2,ab\rangle$, it follows that $L$ acts transitively on $V(\Gamma)$, and so $L$ is a regular subgroup of $A$. This completes the proof.
\end{proof}
%%%%%%%%%%%%%%%%%%%%%%%%%%%%%%%%%%%%%%%%%%%%%%%%%%%%%%%%%%%%%%%%%%%%
For a positive integer $m$, a group $G$ is said to be a connected $m$-(D)CI-group if every connected Cayley (di)graph of $G$ with valency at most $m$ is a CI-(di)graph. If $n$ is odd then, by  \cite[Theorem 3.5]{Qu}, $\mathrm{D}_{2n}$ is a $3$-DCI-group,
and so $\mathrm{D}_{2n}$ is a connected $4$-DCI-group by Theorem~\ref{connected-dihedral}. If $n\geq 4$ is even then $\mathrm{D}_{2n}$ is not a connected $4$-CI-group by Lemma~\ref{even}. Thus we have the following lemma.

\begin{lem}\label{connected 4-DCI} Let $n\geq 3$ be an integer.
% and let $\mathrm{D}_{2n}$ be a dihedral group of order $2n$.
Then $\mathrm{D}_{2n}$ is a connected $4$-DCI-group if and only if $\mathrm{D}_{2n}$ is a connected $4$-CI-group if and only if $n$ is odd.
\end{lem}

%%%%%%%%%%%%%%%%%%%%%%%%%%%%%%%%%%%%%%%%%%%%%%%%%%%%%%%%%%%%%%%%%%%%%%%%%%%%%%%%%%%%%%%%%%%%%%%%%%%%%%%%%%%%%%%%%%%
\smallskip
Now we are ready to prove Theorem~\ref{mainth1}.

\begin{proof}[Proof of Theorem~{\rm \ref{mainth1}}.]
First we prove part~(i). If $\mathrm{D}_{2n}$ has the $4$-DCI property, then Proposition~\ref{p-odd-DCI} asserts that $n$ is odd and is indivisible by $9$. Conversely, suppose that $n$ is odd and is indivisible by $9$. Let $S,\,T\subseteq \mathrm{D}_{2n}\setminus\{1\}$ with $|S|=4=|T|$ and $\Cay(\mathrm{D}_{2n},S)\cong \Cay(\mathrm{D}_{2n},T)$. Then we need to prove that $S\equiv T$. Note that  $\Cay(\mathrm{D}_{2n},S)$ and $\Cay(\mathrm{D}_{2n},T)$ have connected components $\Cay(\langle S\rangle,S)$ and $\Cay(\langle T\rangle,T)$, respectively. It follows that $\Cay(\langle S\rangle,S)\cong \Cay(\langle T\rangle,T)$. In particular, $|\langle S\rangle|=|\langle T\rangle|$. If $|\langle S\rangle|$ is odd then $\langle S\rangle$ and $\langle T\rangle$ are both cyclic and contained in $\langle a\rangle$, and thus $\langle S\rangle=\langle T\rangle$; in this case, we have $S\equiv T$ by (i) of Proposition \ref{4-cyclic} and Lemma \ref{homogeneous}, where the condition $9\nmid n$ applies (this is the only place to use the condition, and under this condition, the digraph is not a graph).
Suppose that $|\langle S\rangle|$ is even. Then we obtain that $\langle S\rangle=\langle a^s,a^tb\rangle$ and $\langle T\rangle=\langle a^s,a^rb\rangle$ for some integers $s$, $t$ and $r$. Let $\alpha=\sigma_{1,r-t}\in \Aut(\mathrm{D}_{2n})$. Then $\langle S^\alpha\rangle=\langle T\rangle$. Note that
\[
\Cay(\langle T\rangle,S^\alpha)=\Cay(\langle S^\alpha \rangle,S^\alpha)\cong \Cay(\langle S\rangle,S)\cong \Cay(\langle T\rangle,T).
\]
By Lemma \ref{connected 4-DCI} and Lemma \ref{homogeneous}, we have $S^\alpha\equiv T$, and so $S\equiv T$. Therefore, Theorem~\ref{mainth1}~(i) holds.

For Theorem~\ref{mainth1}~(ii), the necessity follows from Lemma~\ref{even}. Applying Proposition~\ref{4-cyclic}(ii), Lemma~\ref{connected 4-DCI} and Lemma~\ref{homogeneous}, we derive the sufficiency using a similar argument to the above proof of the sufficiency of Theorem~\ref{mainth1}~(i).
\end{proof}
\medskip

%%%%%%%%%%%%%%%%%%%%%%%%%%%%%%%%%%%%%%%%%%%%%%%%%%%%%%%%%%%%%%%%%%%%
\noindent {\bf Acknowledgements:} We thank Yifan Pei for reading the first draft of this paper and making valuable comments. The work was supported by the National Natural Science Foundation of China (12331013, 12311530692, 12271024, 12161141005) and the 111 Project of China (B16002).
%%%%%%%%%%%%%%%%%%%%%%%%%%%%%%%%%%%%%%%%%%%%%%%%%%%%%%%%%%%%%%%%%%%%
%%%%%%%%%%%%%%%%%%%%%%%%%%%%%%%%%%%%%%%%%%%%%%%%%%%%%%%%%%%%%%%%%%%%

\end{document}